\documentclass[11pt]{amsart}
 \usepackage{graphicx}
 \usepackage{amssymb,latexsym}
\newtheorem{theorem}{Theorem}[section]
\newtheorem{corollary}[theorem]{Corollary}

\newtheorem{lemma}[theorem]{Lemma}

\newtheorem*{remark}{Remark}

\newcommand{\linnum}{\stepcounter{theorem}\tag{\thetheorem}}

\def\G{\Gamma}
\def\F{F\o lner}
\def\t #1.{\tilde{#1}}
\def\Z{\mathbb Z}
\def\zg{\gamma}
\def\R{{\mathbb R}}
\def\bd{\partial}
\def\Rth{{\mathbb R}^3}
\def\Zd{\Delta}
\def\za{\alpha}
\def\zb{\beta}
\def\zg{\gamma}
\def\zd{\delta}
\def\fr #1.#2.{\frac{#1}{#2}}
\def\zl{\lambda}
\def\zep{\epsilon}
\def\inv{^{-1}}
\def\Rn{{\mathbb R}^n}
\def\Rm{{\mathbb R}^m}

\title{Amenability, F\o lner ratios, and cooling functions}
\author{J. W. Cannon}
\email{cannon@math.byu.edu}
\address{Department of Mathematics\\Brigham Young University\\
Provo, UT  84602\\U.S.A.}

\author{W. J. Floyd}
\email{floyd@math.vt.edu}
\address{Department of Mathematics\\Virginia Tech\\
Blacksburg, VA 24061\\U.S.A.
\newline http://www.math.vt.edu/people/floyd}

\author{W. R. Parry}
\email{walter.parry@emich.edu}
\address{Department of Mathematics\\Eastern Michigan University\\
Ypsilanti, MI 48197\\U.S.A.}

\date{\today}

\begin{document}
\maketitle

\begin{abstract} Erling \F\ proved that the amenability or
nonamenability of a countable group $G$ depends on the complexity of
its finite subsets $S \subset G$. Complexity has three measures:
maximum \F\ ratio, optimal cooling function, and minimum cooling norm.

Our first aim is to show that, for a fixed finite subset $S\subset G$,
these three measures are tightly bound to one another. We then explore
their algorithmic calculation.

Our intent is to provide a theoretical background for algorithmically
exploring the amenability and nonamenability of discrete groups.
\end{abstract}

\maketitle

\section{Introduction}\label{sec:intro} A countable group $G$ is
\emph{amenable} if it admits a finitely additive, translation
invariant measure defined on every subset of $G$ such that the measure
of $G$ is 1. Such groups are important in measure theory, dynamical
systems, and ergodic theory. Nonamenable groups lead to the
Hausdorff-Banach-Tarski paradox which asserts that \emph{one} copy of
the unit ball in $\Rth$ can be rigidly torn into finitely many pieces
that can be rigidly reassembled to form \emph{two} copies of the unit
ball.

In \cite{Fol54} and \cite{Fol55}, Erling \F\ gave two
geometric characterizations of countable amenable groups $G$. These
characterizations make these groups very attractive to geometric group
theorists. Since we emphasize finite subsets $S$ of $G$, we shall
state the \F\ theorems in a slightly nonstandard way. In particular,
F\o lner's statement of Theorem 1.2 is in terms of global cooling
functions and not local cooling functions. First we give some
introductory definitions.

Let $\G = \G(G,C)$ denote the Cayley graph of an infinite group $G$
with finite generating set $C$, let $S \subset G$ be nonempty and
finite, let $E(S)$ denote the set of edges of $\G$ with at least one
vertex in $S$, and let $\bd E(S)$ denote the set of boundary edges of
$S$ consisting of those edges that have exactly one vertex in
$S$. Orient each edge $e$ of $E(S)$ with one of its two possible
orientations so that $e$ has an initial vertex $i(e)$ and a terminal
vertex $t(e)$. For the edges $e$ of $\bd E(S)$, choose the orientation
so that $i(e) \in S$.

The \emph{\F\ ratio} of $S$ is the quotient $|S|/|\bd E(S)|$. (We will
see that $\partial E(S)\ne \emptyset$.)  We will never reduce this
fraction except to compare sizes. That is, if we write that the \F\
ratio of $S$ is $a/b$, then we always assume that $a = |S|$ and that
$b = |\bd E(S)|$.

A \emph{cooling function}\ for $S$ is a function $c:E(S) \to \R$ such
that, $\forall \, s \in S,$
$$h(s) \equiv \sum_{i(e) = s} c(e) - \sum_{t(e) = s} c(e) \ge 1.$$ One
can interpret $c(e)$ as the heat pumped along $e$ from $i(e)$ to
$t(e)$. Then $h(s)$ is the net loss of heat at $s$. The \emph{cooling
norm}\ of $c$ is $\left|c\right|=\max_{e \in E(S)}|c(e)|$.

Here are the two \F\ characterizations. For the purposes of this
paper, the reader can take either of these \F\ theorems as defining
amenability.

\begin{theorem}[\F\ set] The finitely generated group $G$ is amenable
if and only if there exist finite subsets $S_1\subset S_2 \subset S_3
\subset \cdots$ exhausting $G$ whose \F\ ratios $|S_i|/|\bd E(S_i)|$
approach $\infty$.
\end{theorem}

\begin{theorem} The finitely generated group $G$ is nonamenable if and
only if there exist finite subsets $S_1\subset S_2 \subset S_3 \subset
\cdots$ exhausting $G$ and admitting cooling functions $c_i:E(S_i)\to
\R$ whose cooling norms $|c_i|$ are uniformly bounded.
\end{theorem}

Though no direct use of measure theory appears in this paper, the
reader may appreciate an orientation to \F's procedures. \F's proofs
proceed roughly as follows.

\F\ first proves his cooling theorem. 

If uniformly bounded cooling functions do not exist, then \F\
generalizes an argument of Banach which describes an outer measure on
the group and applies a version of the Hahn-Banach theorem to create
the desired measure on the group. It is historically interesting to
note that Banach's purpose in proving the Hahn-Banach theorem was to
prove that Abelian groups are amenable.

If uniformly bounded cooling functions do exist, then the Banach
argument fails; and it is fairly easy to show that a measure as in the
definition of amenable group cannot possibly exist.

The \F\ set theorem then has an easy half and a hard half. 

The easy half assumes the existence of subsets whose \F\ ratios
approach $\infty$. Since these sets have relatively small boundaries,
it is impossible to pump a lot of heat out of these sets without using
cooling functions of large norm. Hence cooling functions of uniformly
bounded norm do not exist and the group is amenable.

The hard half assumes that such sets do not exist and, by an extremely
indirect proof, \F\ proves that cooling functions with uniformly
bounded norms do exist.

It is the hard half of this proof that has most interest for us. Our
goal is to make \F's indirect proof as direct and as algorithmic as
possible. In the process, we streamline the \F\ proof.

Statement 2 of Lemma~\ref{lemma:exists} shows that cooling functions
for $S$ always exist in this context.  Using this, it is easy to show
that the two theorems are equivalent by application of our main
theorem:

\begin{theorem} Let $\G$ be a locally finite graph
with vertex set $G$ and a nonempty finite subset $S\subset G$ which
admits a cooling function.  Then $S$ admits a cooling function $c$ of
minimum possible cooling norm $N = |c|$, and
$$N = |c| = \max_{S_0\subset S} |S_0|/ |\bd E(S_0)|.$$
\end{theorem}

That is, the minimum cooling norm is equal to the maximum \F\ ratio of
subsets $S_0$ of $S$, taking the F{\o}lner ratio of the empty set to
be 0.  We call such a cooling function with minimum cooling norm an
\emph{optimal cooling function}.  This absolute cooling theorem will
be a corollary to the cooling theorem of Section~\ref{sec:cooling}.

We prove the cooling theorem in Section~\ref{sec:cooling} by refining
F{\o}lner's original arguments.  In Section~\ref{sec:lp} we interpret
this cooling theorem in terms of linear programming and give another
proof using standard results of linear programming.
Section~\ref{sec:lp} concludes with a brief derivation of statement 2
of the cooling theorem using the max flow min cut theorem.  The linear
programming approach in Section~\ref{sec:lp} leads to an algorithm,
presented in Section~\ref{sec:algo}, which is a modification of the
standard simplex algorithm.  It finds an optimal cooling function $c$
and a F{\o}lner-optimal subset $S_0\subset S$, that is, a subset
$S_0\subset S$ whose F{\o}lner ratio is the maximum possible.
Moreover, the subset $S_0$ which the algorithm finds has the property
that if $S'_0$ is a F{\o}lner-optimal subset of $S$, then $S'_0\subset
S$.  In other words, there is a maximal such subset of $S$ with
maximum F{\o}lner ratio, and the algorithm finds it.

After Section~\ref{sec:algo}, we describe our partial results toward
another algorithmic process for finding the minimum possible cooling
norm $N$, finding an optimal cooling function, and finding a
F{\o}lner-optimal subset $S_0 \subset S$. The main ingredient is the
relative cooling theorem, another corollary to the cooling theorem of
Section~\ref{sec:cooling} (see Section~\ref{sec:relative}). Then we
describe the process of peeling layers away from an arbitrary set to
find a subset of maximum possible \F\ ratio
(Section~\ref{sec:peelings}). And finally we show how to build up
optimal cooling functions layer by layer from known cooling functions
on subsets (Section~\ref{sec:building}).  

It can be challenging to determine whether or not a countable group is
amenable. In particular, despite concerted efforts for over $30$ years
it is still not known whether or not Thompson's group $F$ is
amenable. And while Bartholdi and Vir\'ag showed in \cite{BV05}
that the Basilica group is amenable, their proof is via random walks
and one still doesn't know how to construct a F\o lner sequence for
it. Our hope is that the results of this paper will serve in exploring
amenability for countable groups.

{\bf Permanent setting for the remainder of the paper:} $\G$ is a
locally finite graph; $G$ is the set of vertices of $\G$;
$S$ is a nonempty finite subset of $G$; $E(S)$ is the set of edges of $\G$
having at least one vertex in $S$, and $\bd E(S) \subset E(S)$ is the
set of edges of $\G$ having exactly one vertex in $S$.

{\bf Always} we assume that the set $S$ admits a cooling function.

{\bf Edge orientation:} For each edge $e\in E(S)$, we choose one of
the two possible orientations of $e$, so that $e$ has an \emph{initial
vertex} $i(e)$ and a \emph{terminal vertex} $t(e)$. If $e\in \bd
E(S)$, then we choose that orientation which has $i(e) \in S$ and
$t(e) \notin S$.

{\bf We always assume given,} as initial data, a function $h_0:S \to
(0,\infty)$. In most applications, this function will be identically
equal to $1$, but in the relative cooling theorem it is important that
we be allowed to modify this initial data.

\section{The cooling theorem}\label{sec:cooling}

{\bf Cooling functions:} A \emph{cooling function rel} $h_0$ is a
function $c:E(S) \to \R$ such that
$$\forall \, s\in S,\quad h(s) \equiv \sum_{i(e) = s}c(e) - \sum_{t(e)
= s} c(e) \ge h_0(s).$$We may interpret $c(e)$ as the heat pumped
along edge $e$ from initial vertex $i(e)$ to terminal vertex
$t(e)$. The function $h(s)$ is then the net heat loss at the vertex
$s$. The absolute cooling theorem of the introduction deals with the
special case where $h_0$ is constant and equal to $1$. The value
$h_0(s)$ gives a lower bound on the amount of heat to be pumped out of
vertex $s$. The sum $H(S) =\sum_{s\in S}h_0(s)$ gives the minimal
amount of heat that a cooling function must pump out of $S$ through
the boundary edges $\bd E(S)$ of $S$.

Recall that we always assume that $S$ admits a cooling function.  The
next lemma provides an equivalent condition.  We need two
definitions for this.  Let $T$ be a finite subset of $G$.  A connected
component of $T$ consists of all the vertices in a connected component
of the graph $\cup \{e:e\in E(T)\setminus \partial E(T)\}$.  To say
that a connected component $T_0$ of $T$ has a nonempty boundary, we
mean that $\partial E(T_0)\ne \emptyset$.

\begin{lemma}\label{lemma:exists}  (1)A nonempty finite subset
$T$ of $G$ admits a cooling function rel $h_0$ if and only if every
connected component of $T$ has a nonempty boundary.\par

(2) If $\G$ is infinite and connected, then every nonempty finite 
subset $T\subset G$ admits a cooling function rel $h_0$.
\end{lemma}
  \begin{proof} Let $T$ be a nonempty finite subset of $G$.

We first prove the forward implication of statement 1.  Suppose that
$T$ admits a cooling function rel $h_0$.  Let $T_0$ be a connected
component of $T$.  Then $H(T_0)>0$.  Edges of $E(T_0)\setminus
\partial E(T_0)$ conduct no heat from $T_0$.  Thus $\partial E(T_0)\ne
\emptyset$.  This proves the forward implication of statement 1.

For the backward implication, suppose that every connected component
of $T$ has a nonempty boundary.  Then for each $x\in T$ there exists
an edge path $e_1=(x_0,x_1)$, $e_2 = (x_1,x_2)$, $\ldots$, $e_n=
(x_{n-1},x_n)$ with $x = x_0$, with $x_i \in T$ for $i < n$, and with
$e_n\in \bd E(T)$, so that $x_n \notin T$. Given $x$, transport
$h_0(x)$ units of heat from $x = x_0$ to $x_1$, the same from $x_1$ to
$x_2$, etc., until $h_0(x)$ units have been, in effect, transported
from $x_0$ to $x_n$ with no net gain at any intermediate vertex. That
is, viewing $e_i$ as oriented from $x_{i-1}$ to $x_i$, we define a
function from $E(T)$ to $\R$ so that its value at $e_1,\dotsc,e_n$ is
$h_0(x)$ and 0 otherwise.  Call this function a heat-flow path. Adding
one heat-flow path for each $x \in T$ results in a function $c:E(T)\to
\R$ satisfying $h(x) = h_0(x)$ for each $x \in T$. Thus $c$ is a
cooling function rel $h_0$.

This proves statement 1 of Lemma~\ref{lemma:exists}.

To prove statement 2, suppose that $\G$ is infinite and connected.
Let $x\in T$.  Because $\G$ is infinite and $T$ is finite, there
exists $y\in G\setminus T$.  Because $\G$ is connected, there exists
an edge path in $\G$ joining $x$ and $y$.  Some edge in this edge path
is in $\partial E(T)$.  It follows that every connected component of
the boundary of $T$ is nonempty.  Statement 2 now follows from
statement 1.

This proves Lemma~\ref{lemma:exists}.
\end{proof}

{\bf \F\ ratio:} Because we always assume that $S$ admits a cooling
function rel $h_0$, Lemma~\ref{lemma:exists} implies that $\partial
E(S)\ne \emptyset$.  The \emph{\F\ ratio} $FR(S) =H(S)/|\bd E(S)|$
measures the average amount of heat that must be pumped out of $S$
along each edge $e\in \bd E(S)$ by any cooling function rel $h_0$.
Since $S$ admits a cooling function rel $h_0$, so does every nonempty
subset $S_0$ of $S$.  So $FR(S_0)$ exists for each $S_0$.  We set
$FR(\emptyset)=0$.

\begin{lemma}\label{lemma:lowerbd}  If $S_0\subset S$, then
$FR(S_0)$ is a lower bound on the norm of every cooling function on
$S$ rel $h_0$.
\end{lemma}

\begin{proof} Since a cooling function must cool each subset $S_0
\subset S$, since the heat in $S_0$ must be carried out of $S_0$ along
the boundary edges of $S_0$, and since at least one boundary edge must
carry at least the average required per boundary edge, it follows
immediately that the minimum cooling norm $N = |c|$ of a cooling
function must be at least as large as the \F\ ratio
$FR(S_0)$. 
\end{proof}

Our absolute cooling theorem will show that $N$ is exactly equal to
the maximum of the F{\o}lner ratios which occur in this lemma.

{\bf The simplex $\Zd(S)$ with vertices $S$ and the functions $\t
f.$:} Let $\Zd(S)$ denote the abstract simplex whose vertices are the
elements of the set $S$. An element of $\Zd(S)$ is therefore a
function $f:S \to [0,1]$ such that $\sum_{s\in S}f(s) = 1$. With each
element $f\in \Zd(S)$ we associate a function $\t f.:V\to [0,\infty)$,
where
$$V=\cup_{e \in E(S)}\bd e,$$
$$\forall\, v\in S,\quad\t f.(v) = f(v)/h_0(v), \hbox{ and }$$
$$\forall\, v\in V \setminus S, \quad\t f.(v) = 0.$$

\begin{theorem}[Cooling]\label{thm:cooling} (1) If $N$ is the minimum
possible norm of a cooling function for $S$ rel $h_0$, then the
reciprocal $1/N$ is given by the minimum value of a convex function
$A(f)$ on the simplex $\Zd(S)$ as follows:
$$1/N = \min_{f\in \Zd(S)}A(f),$$ where $\t f.$ is the modification of
$f$ defined above and $$A(f) =\sum_{e\in E(S)}|\t f.(i(e)) - \t
f.(t(e))|.$$

(2) The minimum possible norm $N$ is given by the maximum \F\ ratio:
$$N = \max_{S_0 \subset S}FR(S_0).$$

(3) If $FR(S_0)$ realizes the maximum in (2), then the function $f\in
\Zd(S)$ whose modification $\t f.$ is constant on $S_0$ and $0$ on
$V\setminus S_0$ realizes the minimum in (1).

\end{theorem}

\begin{proof} (1): The verification that $A$ is a convex function is
left to the reader.  Since the function $A:\Zd(S) \to (0,\infty)$ is
positive and continuous on the compact set $\Zd(S)$, there is
certainly a function in $\Zd(S)$ which gives a positive minimum for
$A$.

For the rest, we follow \F's argument. We fix a candidate
norm $N$ for a cooling function rel $h_0$. We introduce a variable
$y_e$ for each $e\in E(S)$ and note that a cooling function of
norm $\le N$ is an assignment of a real value to each of the variables
$y_e$ in such a way that the following inequalities are satisfied:
$$\forall\, s\in S, \quad h(s) = \sum_{i(e)=s}y_e - \sum_{t(e)=s}y_e
\ge h_0(s); \hbox{ and }$$
$$\forall\,e\in E(S), \quad y_e \ge -N \quad \hbox{and }\quad -y_e \ge
-N.$$The first of these inequalities says that each element of $S$ is
cooled by the required amount. The last two inequalities say that the
norm of the resulting cooling function is $\le N$.

We assume that $N$ is too small, so that no cooling function exists,
and conclude from Farkas' lemma, Lemma~\ref{lemma:farkas}, that there
exist nonnegative numbers $\za(s)$, $\zb(e)$, and $\zg(e)$ satisfying
the following two conditions:
  \begin{equation*}\linnum\label{lin:line5}
\sum_{s\in S}\za(s)h(s) +\sum_{e\in E(S)}(\zb(e)-\zg(e))y_e \equiv 0,
  \end{equation*}
and
  \begin{equation*}\linnum\label{lin:line6}
\sum_{s\in S} \za(s)h_0(s) \,\,-\,\, N\cdot\sum_{e\in
E(S)}(\zb(e) + \zg(e)) > 0.
  \end{equation*}
We cannot have $\za(s)$ identically $0$, for then the left-hand side
of line~\ref{lin:line6} would be $\le 0$, a contradiction. We may
therefore scale the coefficients $\za(s)$, $\zb(e)$, and $\zg(e)$ so
that $\sum \za(s)h_0(s) = 1$. That is, we may assume that $f =\za\cdot
h_0 \in \Zd(S)$. We may assume the function $\za$ is extended to the
vertices of $\bd E(S)$ not in $S$ so as to be $0$ on vertices not in
$S$. It follows that this extended $\za$ is our standard modification
$\t f.$ of the function $f$.

We normalize $\zb(e)$ and $\zg(e)$ as follows: We replace the larger
of $\zb(e)$ and $\zg(e)$ by the nonnegative difference $|\zb(e) -
\zg(e)|$ and the other by $0$. This has no effect at all on the
condition in line~\ref{lin:line5}, and it can only increase the sum in
line~\ref{lin:line6}; so the two conditions are still satisfied.

We claim that $\zb(e) + \zg(e) = |\za(i(e)) - \za(t(e))| = |\t
f.(i(e)) - \t f.(t(e))|$. Indeed, the coefficient of $y_e$ in
line~\ref{lin:line5} is 0, and so
  \begin{equation*}\linnum\label{lin:line7}
\za(i(e)) - \za(t(e)) + \zb(e) - \zg(e)=0.
  \end{equation*}
Because of our normalization, one of $\zb(e)$ and $\zg(e)$ is $0$, and
so $\left|\zb(e)+\zg(e)\right|=\left|\zb(e)-\zg(e)\right|$. Hence
$$\zb(e)+\zg(e) = |\zb(e)+\zg(e)|=|\zb(e)-\zg(e)|=|\za(i(e)) -
\za(t(e))|,$$as desired.

With $\sum \za(s)h_0(s) = 1$ and $\zb(e) + \zg(e) = |\za(i(e)) -
\za(t(e))|= |\t f.(i(e)) - \t f.(t(e))|$, we find that
line~\ref{lin:line6} is equivalent to
$$A(f) =\sum_{e\in E(S)} |\t f.(i(e)) - \t f.(t(e))| < 1/N.$$Thus cooling
functions of norm $N$ exist if $1/N \le
\min_{f\in\Zd(S)}A(f)$. 

The converse of this statement can be proved by assuming that there
exists $f\in \Zd(S)$ with $A(f)<1/N$ and reversing this argument.  The
function $f$ determines the numbers $\za(s)$ as above.  We choose the
numbers $\zb(e)$ and $\zg(e)$ as above so that one of them is 0 and so
that line~\ref{lin:line7} is satisfied.  We conclude that
lines~\ref{lin:line5} and \ref{lin:line6} hold.  Now Farkas' lemma
implies that there is no cooling function with norm $N$.  So cooling
functions with norm $N$ exist if and only if $1/N\le \min_{f\in
\Zd(S)}A(f)$.

This proves (1).

(2) and (3): Lemma~\ref{lemma:lowerbd} shows that the minimum possible
cooling norm $N$ for $S$ must be at least as large as
$\max_{S_0\subset S} FR(S_0)$. Our remaining task is to find a subset
$S_0$ of maximum \F\ ratio and to show that the function $g\in \Zd(S)$
associated with $S_0$ as in (3) does in fact realize the minimum $A(g)
= \min_{f\in \Zd(S)} A(f)$ in (1).

Suppose that $A$ assumes its minimum at $f\in \Zd(S)$.  We let
$S_0\subset S$ denote the set of points in $S$ at which the
corresponding function $\t f.$ takes its (positive) maximum. We shall
see that this set $S_0$ satisfies the required conditions.

As in (3), we let $\t g.$ be constant and positive, equal to $\zd$, on
$S_0$, and constant, equal to $0$, on the complement of $S_0$. The
corresponding element $g\in \Zd(S)$ therefore satisfies the equation
$$1 =\sum_{s\in S} g(s) = \sum_{s\in S_0} \zd\cdot h_0(s) = \zd\cdot
H(S_0).$$ Thus $\zd = 1/H(S_0)$.  It is easy to calculate $A(g)$ since
$|\t g.(i(e)) - \t g.(t(e))|$ is only nonzero for $e\in \bd
E(S_0)$. Therefore,
  \begin{equation*}\linnum\label{lin:Ag}
A(g) = \sum_{e\in E(S)}|\t g.(i(e)) - \t g.(t(e))| = \zd\cdot |\bd E(S_0)| =
\fr |\bd E(S_0)|.H(S_0). = \fr 1.FR(S_0)..
  \end{equation*}
We know that $A(f)$ realizes the minimum in (1). We complete the proof
by showing that $A(f) \ge A(g)$ so that the latter also realizes the
minimum in (1).

We apply the methods of the calculus of variations. We consider the
functions $f_\zl\in \Zd(S)$ defined as follows:
$$f_\zl = \fr f - \zl g.1-\zl., \hbox{ for }\zl\text{ very small in }
[0,1).$$Note that we are simply reducing some of the positive values
of $f$ by a little bit and are then scaling so as to remain in
$\Zd(S)$. The following lemma completes the proof.

\begin{lemma}$$0 \le \fr d(A(f_\zl)).d\zl.\bigg|_{\zl = 0} =  A(f) -A(g) .$$
\end{lemma}

\begin{proof} The derivative is certainly $\ge 0$ since $A(f)$ minimizes $A$. 

It is an easy matter to take the derivative of $|x|$ if we know which
of the two options $|x| = 1 \cdot x$ or $|x| = -1\cdot x$ is
true. Luckily, we can determine appropriate signs $\pm 1$ for each of
our three absolute values $|\t f.(i(e)) - \t f.(t(e))|$, $|\t
f_\zl.(i(e)) - \t f_\zl.(t(e))|$, and $|\t g.(i(e)) - \t g.(t(e))|$,
which appear in the defining formula for $A(f)$, $A(f(\zl))$, and
$A(g)$, and those signs may be chosen compatibly.

To each edge $e\in E(S)$ we assign a number $\zep(e) = \pm 1$ as
follows: if $|\t f.(i(e)) - \t f.(t(e))| = \t f.(i(e)) - \t f.(t(e))$,
then $\zep(e) = 1$; otherwise, $\zep(e) = -1$.

For $\zl$ sufficiently small, the function $\t f_\zl.$ will clearly
require the same signs $\zep(e)$.

Because the function $\t g.$ is positive only on the elements of
$S_0$, where $\t f.$ takes on its maximum value, $\t g.$ will require
the same sign provided that exactly one end point of $e$ is in
$S_0$. If neither end point of $e$ is in $S_0$ or if both end points
of $e$ are in $S_0$, then $\t g.(i(e)) - \t g.(t(e)) = 0$ and we may
use the same value of $\zep(e)$ for $\t g.$ as has already been chosen
for $\t f.$ and $\t f_\zl.$.

Thus we may choose multipliers $\zep(e) = \pm 1$ so that we have the
following three equalities for all $\zl$ sufficiently small:
$$|\t f.(i(e)) - \t f.(t(e))| = \zep(e)\cdot\big(\t f.(i(e)) - \t
f.(t(e))\big)$$
$$|\t f_\zl.(i(e)) - \t f_\zl.(t(e))| = \zep(e)\cdot\big(\t
f_\zl.(i(e)) - \t f_\zl.(t(e))\big)$$
$$|\t g.(i(e)) - \t g.(t(e))| = \zep(e)\cdot\big(\t g.(i(e)) - \t
g.(t(e))\big).$$We would be unable to obtain compatible signs $\zep(e)$
if $S_0$ had been chosen via anything but the \emph{maximum} value of
$\t f.$. With all absolute value signs replaced by constants
$\zep(e)$, it is an easy matter to calculate $d(A(f_\zl))/d\zl$ by the
quotient rule, and the result is as stated in the lemma:

$$\vbox{\halign{ $#$\hfil&$#$\hfil\cr d(A(f_\zl))/d\zl\bigg|_{\zl =
0} &=\sum_e\zep(e)\big(\t f.(i(e)) - \t f.(t(e))\big)\fr
d.d\zl.\bigg(\fr 1.1-\zl.\bigg)\bigg|_{\zl = 0}\cr &\hskip .2in
-\sum_e\zep(e)\big(\t g.(i(e)) - \t g.(t(e))\big)\fr d.d\zl.\bigg(\fr
\zl.1 - \zl.\bigg)\bigg|_{\zl = 0}\cr &= A(f) - A(g).\cr }}$$
\end{proof}

This argument completes the proof of the cooling theorem.
\end{proof}

\begin{proof}[Proof of the absolute cooling theorem] Set $h_0 \equiv
1$.
\end{proof}

We conclude this section with a discussion of Farkas' lemma.  In
Section 4 of \cite{Fol55}, F{\o}lner states essentially the
following result, which he presents as a minor modification of a
result proved by Carver in Theorem 3 of \cite{Car22}.

\begin{lemma}[Farkas' lemma]\label{lemma:farkas}  Let $L:\Rn
\to \Rm$ be a linear map and let $b\in \Rm$. Then a necessary and
sufficient condition that there be no solution to the inequality $L(x)
\ge b$ is that there exist a nonnegative vector $\za \in \Rm$, $\za
\ge 0$, such that $L(x)\cdot \za \equiv 0$ but $b \cdot \za > 0$.
\end{lemma}

This is one of the equivalent formulations of Farkas' lemma.  To see
this, we apply Corollary 7.1e on page 89 of Schrijver's book
\cite{Sch99}.  This corollary is one form of Farkas' lemma.  It
states the following, a bit loosely.  Let $A$ be a matrix and let $b$
be a vector.  Then the system $Ax\le b$ of linear inequalities has a
solution $x$, if and only if $yb\ge 0$ for each row vector $y\ge 0$
with $yA=0$.  Replacing $b$ with $-b$, yields the following.  The
system $Ax\ge b$ of linear inequalities has a solution $x$, if and
only if $yb\le 0$ for each row vector $y\ge 0$ with $yA=0$.
Lemma~\ref{lemma:farkas} is the contrapositive of this.

\section{Linear programming }\label{sec:lp}

In this section we give another proof of the cooling theorem,
Theorem~\ref{thm:cooling}, using standard results from linear
programming.  Murty \cite{Mur76} and Schrijver
\cite{Sch99} are good references for linear programming.  We
maintain the notation of Section~\ref{sec:cooling}.

We wish to minimize the maximum value of
  \begin{equation*}
c(e),-c(e)\qquad\forall e\in E(S)
  \end{equation*}
subject to the conditions
  \begin{equation*}
\sum_{i(e)=s}^{}c(e)-\sum_{t(e)=s}^{}c(e)\ge h_0(s)\qquad\forall
s\in S.
  \end{equation*}
In the spirit of linear programming, we introduce a variable $x_e$ for
every $e\in E(S)$ and restate this problem as follows.  Minimize the
maximum value of
  \begin{equation*}
x_e,-x_e\qquad\forall e\in E(S)
  \end{equation*}
subject to the conditions
  \begin{equation*}
\sum_{i(e)=s}^{}x_e-\sum_{t(e)=s}^{}x_e\ge h_0(s)\qquad\forall s\in
S.
  \end{equation*}
Now we introduce one more variable $x_h$ and in effect replace $x_e$
by $x_e$ divided by the maximum value of $\{\left|x_e\right|:e\in E(S)\}$
to see that our problem is equivalent to the following problem.
Maximize $x_h$ subject to the conditions
  \begin{equation*}\linnum\label{lin:system}
\begin{gathered}
h_0(s)x_h-\sum_{i(e)=s}^{}x_e+\sum_{t(e)=s}^{}x_e\le 0\qquad\forall
s\in S\\
\left.\begin{matrix}x_e\le 1 \\-x_e\le 1\end{matrix}\right\}\qquad\forall
e\in E(S).
\end{gathered}
  \end{equation*}
The solution of this problem is the inverse of the solution of our
original problem.  

This is a linear programming problem in the form $\max\{cx:Ax\le b\}$
as in line 13 on page 90 of Schrijver's book \cite{Sch99}.
Here $x$ is the column vector of variables $x_h$, $x_e$ for $e\in
E(S)$.  We have that $c$ is the row vector with as many components as
$x$ with the component corresponding to $x_h$ being 1 and all other
components being 0.  Similarly, $b$ is the column vector of right side
constants in line~\ref{lin:system} and $A$ is the coefficient matrix
of this system of linear inequalities.  (The $A$ here is not to be
confused with our previous $A$.)

Now we apply the duality theorem of linear programming.  Just as
Farkas' lemma, which we encountered at the end of
Section~\ref{sec:cooling}, has a number of formulations, so does the
duality theorem.  We apply the formulation given in Corollary 7.1g
on page 90 of Schrijver's book \cite{Sch99}:
  \begin{equation*}
\max\{cx:Ax\le b\}=\min\{yb:y\ge 0,yA=c\},
  \end{equation*}
provided that both of these sets are nonempty.  The origin is in the
first set, so the first set is nonempty.  We will soon see that the
second one is nonempty too.

So we now have new variables $y_s$ for every $s\in S$, one for every
inequality in line~\ref{lin:system} corresponding to an element of
$S$.  We also have new variables $y_e$ and $z_e$ for every $e\in E(S)$
corresponding to the two inequalities in line~\ref{lin:system} for
every $e\in E(S)$.  Furthermore $y$ is the row vector whose entries are
these new variables.  The duality theorem transforms our problem into
the following one.  Minimize
  \begin{equation*}
\sum_{e\in E(S)}^{}(y_e+z_e)
  \end{equation*}
subject to the conditions
  \begin{equation*}
\begin{gathered}
\sum_{s\in S}^{}h_0(s)y_s=1 \\
-y_{i(e)}+y_{t(e)}+y_e-z_e=0\qquad \forall e\in E(S)\\
y_s\ge 0\quad\forall s\in S\quad y_e\ge 0\quad\forall e\in E(S)
\quad z_e\ge 0\quad\forall e\in E(S).
\end{gathered}
  \end{equation*}
The last equation requires the convention that $y_{t(e)}=0$ if
$t(e)\ne S$.

  Suppose that we fix the variables $y_s$ and let the variables $y_e$
and $z_e$ vary.  Then $y_e-z_e$ is fixed for every $e\in E(S)$.  If $y$
and $z$ are nonnegative real numbers with $y-z=a$, a fixed value, then
the smallest possible value for $y+z$ is $\left|a\right|$.  Thus our
problem is equivalent to the following one.  Minimize
  \begin{equation*}
\sum_{e\in E(S)}\left|y_{i(e)}-y_{t(e)}\right|
  \end{equation*}
subject to the conditions
  \begin{equation*}
\begin{gathered}
\sum_{s\in S}^{}h_0(s)y_s=1 \\
y_s\ge 0\quad\forall s\in S.
\end{gathered}
  \end{equation*}
Now it is clear that our second set is also nonempty, and we have
statement 1 of the cooling theorem.

To prove statements 2 and 3 of the cooling theorem, note that the map
$A$ is not only convex but piecewise linear.  In fact, its restriction
to every cell of the first barycentric subdivision of $\Zd(S)$ is
linear.  As for general linear programming problems, it follows that
its minimum occurs at a vertex of the first barycentric subdivision of
$\Zd(S)$.  (See Section 3.5.5 of Murty's book \cite{Mur76} or
Section 8.3 of Schrijver's book \cite{Sch99}.)  These vertices
are the points $f\in \Zd(S)$ such that $f$ is constant on some subset
$S_0$ of $S$ and 0 on $S\setminus S_0$.  From the definition of $A$,
we see that if $f$ is such a point of $\Zd(S)$, then
  \begin{equation*}
A(f)=\frac{|\partial E(S_0)|}{|S_0|}=\frac{1}{FR(S_0)},
  \end{equation*}
as in line~\ref{lin:Ag}.

This completes our linear programming proof of the cooling theorem.

We conclude this section with a brief discussion involving the max
flow min cut theorem.  Given the close connection between the max flow
min cut theorem and linear programming, (See Section 7.10 of
Schrijver's book \cite{Sch99}.) it should not be surprising
that the max flow min cut theorem is relevant to our problem.

We apply Theorem 1.1 on page 38 of Ford and Fulkerson's book
\cite{FF62} We take the sets $R$ and $S$ there to be empty and
the set $T$ there to be our set $S$.  The set $\overline{X}$ there is
our $S_0$.  The function $f$ there is our cooling function $c$.  The
function $b$ there is our $h_0$.  The function $c$ there is our norm
$n=\left|c\right|$.  The conclusion is that the real number $n$ is the
norm of a cooling function $c$ on $S$ rel $h_0$ if and only if
$H(S_0)\le n \left| \partial E(S_0)\right|$ for every $S_0\subset S$.
Thus $N=\max_{S_0\subset S}FR(S_0)$.  This is statement 2 of the
cooling theorem.

\section{The modified simplex algorithm }\label{sec:algo}

The previous section shows that the problem of finding an optimal
cooling function is a linear programming problem.  As such, it can be
solved by the simplex method.  In our situation every step of the
simplex method proceeds as follows.  We have a cooling function $c$,
and the simplex method finds the vertices at which $c$ is not
optimal.  One chooses one of these vertices, and the step concludes
with a computation which improves $c$ at the chosen vertex.  This
algorithm has the defects that at every step the norm of $c$ usually
does not decrease, there is this choice of vertex, and poor choices
might lead to an infinite loop, although a further enhancement of the
algorithm can avoid infinite loops.  

The algorithm which we present in this section is a greedy algorithm
in that at every step we improve our cooling function $c$ at every
vertex at which $c$ is not optimal.  Because we deal with every bad
vertex, the norm of $c$ decreases at every step.  In our algorithm the
numerator of $\left|c\right|$ is always bounded by $\left|S\right|$
and the denominator of $\left|c\right|$ is bounded by
$\left|E(S)\right|$.  Since $\left|c\right|$ decreases at every step,
the number of steps is therefore bounded by
$\left|S\right|\left|E(S)\right|$.  The number of operations per step
is linear in $\left|E(S)\right|$.  Thus the number of operations
required for our algorithm to find the optimal cooling function is
cubic in $\left|E(S)\right|$.  In practice it seems to be quadratic.
The steps in our algorithm are far more complicated than the steps of
the simplex method, but there are far fewer of them.  In our limited
experience, our algorithm is faster.  Moreover, the subset $S_0$ of
$S$ with maximal F{\o}lner ratio which our algorithm finds also has
the property (See Theorem~\ref{thm:maxl} and the discussion
immediately preceding it.) that if $T$ is a subset of $S$ with maximal
F{\o}lner ratio, then $T\subset S_0$.

Here is our modified simplex algorithm.

Let $\G$ be a locally finite graph, as usual.  The edges of $\G$ are
initially undirected.  However, we often direct the edges of $\G$,
choosing directions to suit the occasion.  If $e$ is a directed edge
of $\G$, then, as usual, we let $i(e)$ denote the initial vertex of
$e$, and we let $t(e)$ denote the terminal vertex of $e$.

Now we fix a nonempty finite set $S$ of vertices of $\G$, as usual.
We direct the edges of $E(S)$ so that every edge of $\partial E(S)$ is
directed away from $S$.  Let $h$ be a nonnegative real number.  A
cooling function for $S$ relative to $h$ is a function $c: E(S)\to \R$
such that
  \begin{equation*}
\sum_{i(e)=v}^{}c(e)-\sum_{t(e)=v}^{}c(e)\ge h\quad\text{ for every }v\in S.
  \end{equation*}
We refer to this inequality as the vertex condition or vertex
inequality at $v$.

The absolute cooling theorem deals with the case in which $h=1$.  It
states that $S$ admits a cooling function relative to $h=1$ of minimum
norm, and this norm is $N=\underset{S_0\subset
S}{\max}FR(S_0)$.  In other words, there exists a maximum value
of $h$ such that $S$ admits a cooling function relative to $h$ with
norm 1, and this maximum value of $h$ is
$N^{-1}=\underset{S_0\subset S}{\min}FR(S_0)^{-1}$.  In this
way we are led to the problem of maximizing $h$ over all cooling
functions for $S$ with norm 1.  This is what our modified simplex
algorithm does: it maximizes $h$ over all cooling functions for $S$
with norm 1.  For every edge $e\in E(S)$, we refer to the inequality
$|c(e)|\le 1$ as the edge condition or edge inequality at $e$.

Our modified simplex algorithm proceeds in steps, starting with step
0.  After $n$ steps we have the following.  We have a forest $F_n$,
which is a subgraph of $\G$.  The vertex set of $F_n$ equals $S$.
Every connected component of $F_n$ is a rooted tree.  We direct every
edge of $F_n$ down toward the root of its component.  We also have a
set $R_n\subset E(S)\setminus \partial E(S)$.  No element of $R_n$ is
an edge of $F_n$, and every element of $R_n$ is directed.  We direct
the edges of $E(S)$ compatibly with $F_n$, $R_n$ and $\partial E(S)$.
We also have a nonnegative real number $h_n$.  A cooling function for
$(S,F_n,R_n,h_n)$ is a function $c: E(S)\to \R$ satisfying the
following conditions.
\begin{enumerate}
  \item $c(e)=1$ if $e\in R_n\cup \partial E(S)$
  \item $c(e)=0$ if $e$ is neither an edge of $F_n$ nor in $R_n\cup
\partial E(S)$
  \item $\sum_{i(e)=v}^{}c(e)-\sum_{t(e)=v}^{}c(e)=h_n$ for every
element $v$ of $S$ which is not a root of $F_n$
  \item $\sum_{i(e)=v}^{}c(e)-\sum_{t(e)=v}^{}c(e)\ge h_n$ for every
root $v$ of $F_n$
\end{enumerate}
In addition to the above, we have a cooling function $c_n$ for
$(S,F_n,R_n,h_n)$ with norm 1 which never takes the value
$-1$. Finally, $h_n$ is maximal with respect to the property that
there exists a cooling function for $(S,F_n,R_n,h_n)$ with norm 1.
Given such a quadruple $(F_n,R_n,h_n,c_n)$, the algorithm constructs
another one $(F_{n+1},R_{n+1},h_{n+1},c_{n+1})$ with $h_{n+1}>h_n$
unless the algorithm finds a subset of $S$ with maximum F{\o}lner
ratio.  This eventually occurs, the algorithm finds a subset of $S$
with maximum F{\o}lner ratio $h_n^{-1}$ together with an associated
cooling function and the algorithm stops.

The situation at step 0 is as simple as possible.  The forest $F_0$
has no edges.  Its connected components, the elements of $S$, are
trivial rooted trees.  The set $R_0$ is empty.  The cooling function
$c_0$ takes the value 1 on $\partial E(S)$ and the value 0 elsewhere.
The maximality of $h_0$ then implies that
  \begin{equation*}
h_0=\underset{v\in S}{\min}|\{e\in \partial E(S):i(e)=v\}|.
  \end{equation*}
All conditions are satisfied.

Now let $n$ be a nonnegative integer, and suppose that we have $F_n$,
$R_n$, $h_n$ and $c_n$ as above.  We prepare to construct $F_{n+1}$,
$R_{n+1}$, $h_{n+1}$ and $c_{n+1}$ in the next paragraph.

Suppose that the edge inequality satisfied by $c_n$ is strict for
every edge of $F_n$.  Let $T$ be a nontrivial connected component of
$F_n$, and let $v$ be a leaf of $T$.  (Roots are not leaves.)  Because
the edge inequality satisfied by $c_n$ at the edge $e$ of $T$ which
contains $v$ is strict, it is possible to send slightly more than
$c_n(e)$ units of heat toward the root of $T$ along $e$ while
satisfying the vertex equality at $v$ for a value slightly larger than
$h_n$.  We modify $c_n$ in this way at every edge of $T$ to satisfy
the vertex equality at every vertex of $T$ other than its root for a
real number slightly larger than $h_n$.  This applies to every such
connected component of $F_n$.  Thus because $h_n$ is maximal, the
vertex inequality satisfied by $c_n$ for $h_n$ at some root of $F_n$
is actually an equality.  We conclude that either the edge inequality
satisfied by $c_n$ at some edge of $F_n$ is actually an equality or
the vertex inequality satisfied by $c_n$ for $h_n$ at some root of
$F_n$ is actually an equality.  We observe that because $c_n$ never
takes the value $-1$, if $e$ is an edge of $F_n$, then $|c_n(e)|=1$ if
and only if $c_n(e)=1$.

We digress briefly in this paragraph to show that $h_n^{-1}$ is a
relative F{\o}lner ratio.  According to the previous paragraph there
exists either a root or an edge of $F_n$ such that the inequality
satisfied by $c_n$ at either this root or edge is actually an
equality.  If this equality holds at a root, then let $v$ be this
root, and if this equality holds at an edge, then let $v$ be the upper
vertex of this edge.  Let $S_0$ be the set of those vertices in $S$
which are either equal to or above $v$ relative to $F_n$.  By
assumption the vertex inequality satisfied by $c_n$ for $h_n$ is
actually an equality at every element of $S_0$ other than $v$, and now
this is even true at $v$.  We combine these equations, one for every
vertex of $S_0$.  We use the facts that $c_n(e)=1$ for every edge
$e\in R_n\cup \partial E(S)$ and $c_n(e)=0$ if $e$ is neither an edge
of $F_n$ nor in $R_n\cup \partial E(S)$.  We also use the fact that if
$v$ is not a root, then the edge of $F_n$ immediately below it is
directed away from $S_0$ and $c_n$ has value 1 at it.  Letting
$\partial 'E(S_0)$ denote the set of edges of either $\partial E(S)$
or $F_n$ which contain exactly one element of $S_0$, we obtain that
  \begin{equation*}
|\partial' E(S_0)|+\sum_{\substack{e\in R_n\\i(e)\in
S_0}}^{}1-\sum_{\substack{e\in R_n\\t(e)\in S_0}}^{}1=|S_0|h_n.
  \end{equation*}
Hence $h_n=RFR(S_0)^{-1}$, where
  \begin{equation*}
RFR(S_0)=|S_0|\left(|\partial' E(S_0)|+
\sum_{\substack{e\in R_n\\i(e)\in
S_0}}^{}1-\sum_{\substack{e\in R_n\\t(e)\in S_0}}^{}1\right)^{-1}.
  \end{equation*}
We view $RFR(S_0)$ as a relative F{\o}lner ratio.  It is the
F{\o}lner ratio of $S_0$ relative to $S$, $F_n$ and $R_n$.

We return to the result of the penultimate paragraph.  Either the edge
inequality satisfied by $c_n$ at some edge of $F_n$ is actually an
equality or the vertex inequality satisfied by $c_n$ for $h_n$ at some
root of $F_n$ is actually an equality.  We find every such root and
edge.  Every such root loses its status as root. Every such edge, but
not its vertices, is removed from $F_n$.  The result is a forest,
which we denote by $F'_n$.  Every edge just removed from $F_n$,
directed as in $F_n$, is added to $R_n$.  The result is a set of
directed edges, which we denote by $R'_n$.  The forest $F'_n$ has
connected components without roots.  We now construct $F_{n+1}$ by
inductively enlarging the rooted connected components of $F'_n$.  We
adjoin certain edges to $F'_n$ to construct a maximal forest with at
most one root per connected component.  The edges which we adjoin are
edges of $E(S)\setminus\partial E(S)$ which are either not in $R'_n$
or directed away from the rooted connected components.  So, if
possible we adjoin such an edge to $F'_n$, joining one of its rooted
connected components with one of its unrooted connected components.
The result is a forest, one of whose rooted connected components
contains two connected components of $F'_n$.  If possible we adjoin
such an edge to this new forest, joining one of its rooted connected
components with one of its unrooted connected components. We continue
in this way as long as possible.  This final forest is $F_{n+1}$.

Suppose that $F_{n+1}$ has a connected component which is not rooted
(which occurs if $F'_n$ has no roots).  Let $S_0$ be the set of
elements of $S$ which are not contained in rooted connected components
of $F_{n+1}$.  Then every edge in $\partial E(S_0)$ is either in
$\partial E(S)$ or it is in $R'_n$ and directed away from $S_0$.
Hence $c_n(e)=1$ for every $e\in \partial E(S_0)$.  It follows that
$c_n$ is removing as much heat from $S_0$ as is possible while
maintaining norm 1.  As in the paragraph which discusses relative
F{\o}lner ratios, there is no cooling function for $S$ with norm 1
relative to a value larger than $h_n=RFR(S_0)^{-1}=FR(S_0)^{-1}$,
where this relative F{\o}lner ratio is computed relative to $R'_n$.
In this case $S_0$ is a subset of $S$ with maximum F{\o}lner ratio,
and the algorithm stops.

Suppose that the algorithm does not stop at step $n$.  Then every
connected component of $F_{n+1}$ is rooted.  We define $R_{n+1}$ to be
the set of directed edges gotten from $R'_n$ by deleting those edges
of $R'_n$ which became edges of $F_{n+1}$.  If $v$ is a root of $F_n$
at which the inequality satisfied by $c_n$ for $h_n$ is an equality,
then $v$ is not a root of $F_{n+1}$.  If $e$ is an edge of $F_n$ at
which the inequality satisfied by $c_n$ is an equality, that is,
$c_n(e)=1$, then either $e$ is not an edge of $F_{n+1}$ or it is
directed away from the root of its connected component.  The cooling
function $c_n$ need not be a cooling function for
$(S,F_{n+1},R_{n+1},h_n)$ mostly because some edge directions might
have changed.  By simply changing some signs if necessary, $c_n$
determines a cooling function $c'_n$ for $(S,F_{n+1},R_{n+1},h_n)$
except for possibly taking the value $-1$ at some edges.  The value of
$c'_n$ at every edge of $F_{n+1}$ is strictly less than 1.  As at the
beginning of the passage from step $n$ to step $n+1$, this implies
that by appropriately increasing every value of $c'_n$ by a positive
amount we obtain a cooling function for $(S,F_{n+1},R_{n+1},h_n)$
which satisfies every conditional inequality for a real number
strictly larger than $h_n$.  Because the vertex inequality for $c'_n$
is an equality for every vertex of $F_{n+1}$ which is not a root,
every value of every cooling function for $(S,F_{n+1},R_{n+1},h'_n)$
with $h'_n\ge h_n$ is at least as large as the corresponding value of
$c'_n$.  Maximizing, we obtain a real number $h_{n+1}$ and a cooling
function $c_{n+1}$ for $(S,F_{n+1},R_{n+1},h_{n+1})$ with norm 1 which
never takes the value $-1$ such that $h_{n+1}$ is maximal with respect
to the property that there exists a cooling function for
$(S,F_{n+1},R_{n+1},h_{n+1})$ with norm 1.  The quadruple
$(F_{n+1},R_{n+1},h_{n+1},c_{n+1})$ satisfies all required conditions.
This completes the description of the algorithm.

In this paragraph we show that this algorithm ends with a solution
after finitely many steps.  We have seen that $h_n^{-1}$ is a relative
F{\o}lner ratio for every $n$.  Hence the numerator of $h_n$ is
bounded by $|E(S)|$ and the denominator of $h_n$ is bounded by $|S|$.
There are only finitely many possibilities for $h_n$.  Since the
sequence $h_0,h_1,h_2,\ldots$ is strictly increasing, it follows that
the algorithm ends after finitely many steps.

The proof of the following theorem is based on this modified simplex
algorithm.  By comparing the proof and the algorithm, one sees that
the algorithm finds the subset $S_0$ of the theorem.

\begin{theorem}\label{thm:maxl} The set $S$ contains a maximal 
F{\o}lner-optimal subset $S_0$, that is, a subset $S_0$ with maximum
F{\o}lner ratio, such that every F{\o}lner-optimal subset of $S$ is
contained in $S_0$.
\end{theorem}
  \begin{proof} After choosing directions for the edges of $E(S)$, the
absolute cooling theorem implies that there exists a cooling function
$c: E(S)\to \R$ with norm 1 rel $h_0$, where $h_0$ is the inverse
of the maximum F{\o}lner ratio of a subset of $S$.  Without loss of
generality we assume that if $e\in \partial E(S)$, then $e$ is
directed away from $S$ and that $c(e)=1$.

Let $T$ be a F{\o}lner-optimal subset of $S$.  So $h^{-1}_0=FR(T)$.
We direct every edge of $\partial E(T)$ away from $T$.  By definition
$c$ satisfies the vertex inequality,
  \begin{equation*}
\sum_{i(e)=v}^{}c(e)-\sum_{t(e)=v}^{}c(e)\ge h_0
  \end{equation*}
for every $v\in S$, and so this holds for every $v\in T$.  We combine
these vertex inequalities, one for every element of $T$, and obtain
  \begin{equation*}
\sum_{e\in \partial E(T)}^{}c(e)\ge |T|h_0=|\partial E(T)|.
  \end{equation*}
Since $|c(e)|\le 1$ for every $e\in \partial E(T)$ and the number of
these summands equals $|\partial E(T)|$, we can make two conclusions.
One is that every vertex inequality for $T$ is actually an equality.
The other is that $c(e)=1$ for every $e\in \partial E(T)$.

We use these two conclusions to construct the maximal subset of $S$
with maximum F{\o}lner ratio.  For this we define a (nonreflexive,
nonsymmetric and nontransitive) relation on $S$ as follows.  Given
$u,v\in S$, we have that $u\sim v$ if and only if there exists $e\in E(S)$
directed so that $t(e)=u$, $i(e)=v$ and $c(e)<1$.  (The condition
$c(e)<1$ is equivalent to $c(e)\ne 1$.)  Here we view $c$ as being
defined not only for every edge $e\in E(S)$ with its given direction
but also for the opposite edge $\overline{e}$, so that
$c(\overline{e})=-c(e)$.

Let $X$ be the set of vertices of $S$ whose associated vertex
inequality is strict.  Let $\overline{X}$ be the closure of $X$ under
the above relation.  Conclusion one from above implies that every
subset of $S$ with maximum F{\o}lner ratio is contained in the
complement of $X$.  Conclusion two then implies that every subset of
$S$ with maximum F{\o}lner ratio is contained in the complement of
$\overline{X}$.  

Now let $S_0$ be the complement of $\overline{X}$.  To prove the
theorem it suffices to prove that $FR(S_0)=h_0^{-1}$, the
maximum F{\o}lner ratio of a subset of $S$.  For this we first observe
that every vertex inequality is actually an equality for every element
of $S_0$.  Furthermore, if $e\in \partial E(S_0)$ is directed away
from $S_0$, then $c(e)=1$.  Combining the vertex equalities associated
to $S_0$ as in the second paragraph of this proof, we find that
$|\partial E(S_0)|=|S_0|h_0$.  This implies that
$FR(S_0)=h_0^{-1}$.

This proves Theorem~\ref{thm:maxl}.
\end{proof}

\section{The relative cooling theorem}\label{sec:relative} Our plan
for the rest of the paper is to construct cooling functions by finding
a F{\o}lner-optimal subset $S_0 \subset S$, building an optimal
cooling function $c_0$ on $S_0$, and then building cooling functions
on larger and larger subsets of $S$, extending $c_0$ layer by layer. A
cooling function on a smaller subset pumps heat into the layers not
yet considered. Thus we need to know when the extension is
possible. This problem is dealt with by means of the relative cooling
theorem, which is a corollary to the cooling theorem of
Section~\ref{sec:cooling}.  The only twist in the argument is that we
do not apply the theorem to a subgraph of $\G$ but rather to a
slightly modified subgraph.

{\bf Setting:} We assume $\G$, $G$, $S$, $E(S)$, and $\bd E(S)$ given,
with edges $e\in E(S)$ oriented as before. 

{\bf The modified graph $\G'_S$ and modified vertex set $S'$:} We form
a new graph $\G'_S$ from the graph $\G_S = \cup\{e\,|\, e\in E(S)\}$
by splitting all edges $e\in \bd E(S)$ apart at any common vertex in
the complement of $S$ so that the boundary edges have distinct
terminal vertices in the complement of $S$. Let $S'$ be a subset of the
set of vertices of $\G '_S$ with $S\subset S'$.

{\bf Initial conditions:} We assume given a function $h_0:S' \to
(0,\infty)$.

{\bf We assume:} $S'$ admits a cooling function rel $h_0$.

Lemma~\ref{lemma:exists} implies that $S'$ admits a cooling function
rel $h_0$ if and only if every connected component of $S'$ has a
nonempty boundary.

We immediately obtain the following \emph{relative cooling theorem},
which is simply the general cooling theorem of
Section~\ref{sec:cooling} applied to our modified graph $\G'_S$:

\begin{theorem}[Relative cooling]
(1) If $N$ is the minimum possible norm of a cooling function rel
$h_0$, then the reciprocal $1/N$ is given by the minimum value of a
convex function $A(f)$ on the simplex $\Zd(S')$ as follows:
$$1/N = \min_{f\in \Zd(S')}A(f),$$ where $\t f.$ is the modification
of $f$ defined above and $$A(f) =\sum_{e\in E(S')}|\t f.(i(e)) - \t
f.(t(e))|.$$

(2) The minimum possible norm $N$ is given by the maximum \F\ ratio:
$$N = \max_{S_0 \subset S'}FR(S_0).$$

(3) If $FR(S_0)$ realizes the maximum in (2), then the function $f\in
\Zd(S')$ whose modification $\t f.$ is constant on $S_0$ and $0$ on
$S'\setminus S_0$ realizes the minimum in (1).

\end{theorem}We shall apply this relative cooling theorem in
Section~\ref{sec:peelings}.

\section{Peelings}\label{sec:peelings} We assume $\G$, $G$, $S$,
$E(S)$, $\bd E(S)$, and $h_0:S \to (0,\infty)$ given as before. Our
goal is to find a F{\o}lner-optimal subset $S_0\subset S$. We propose
to do so by peeling layers away from $S$ until we find the desired
set.

Let $P$ be a nonempty subset of $S$, and let $T=S\setminus P$.  We
have the (split) graph $\G'_P$ from the previous section. Let $P'$ be
the subset of $\G'_P$ consisting of $P$ together with the boundary
vertices of $\G'_P$ arising from (the splitting of) points of $T$.  If
$P_0 \subset P'$, then we let $B(P_0)$ denote the cardinality of
$P_0\setminus P$ and we let $B'(P_0)$ denote the number of edges in
$\G'_P$ having exactly one vertex in $P_0$. [Note that if $x\in
P_0\setminus P$ is not joined by its unique edge $e$ in $\G'_P$ to a
point of $P_0$, then $x$ is counted in $B(P_0)$ and $e$ is counted in
$B'(P_0)$.]  We call $P$ a \emph{peeling} for $S$ if
$$\forall\, P_0\subset P',\quad H(P_0\cap P)\le FR(T)(B'(P_0) -
B(P_0)).$$

The next lemma gives an alternate characterization of peelings.

\begin{lemma}\label{lemma:alt}  Let $N$ be a positive real
number.  In the above setting, we extend $h_0|_P:P\to (0,\infty)$ to
$P'$ by defining $h_0(x)=N$ for each $x\in P'\setminus P$.  Then
$FR(P_0)\le N$ if and only if $H(P_0\cap P)\le N(B'(P_0)-B(P_0))$ for
every $P_0\subset P'$.  In particular, taking $N=FR(T)$ shows that $P$
is a peeling for $S$ if and only if $FR(P_0)\le FR(T)$ for every
$P_0\subset P'$.
\end{lemma}
  \begin{proof} Let $P_0\subset P'$.  Then
  \begin{equation*}
FR(P_0)=\frac{N\cdot B(P_0)+H(P_0\cap P)}{B'(P_0)}.
  \end{equation*}
So $FR(P_0)\le N$ if and only if
  \begin{equation*}
N\cdot B(P_0)+H(P_0\cap P)\le N\cdot B'(P_0)
  \end{equation*}
if and only if
  \begin{equation*}
H(P_0\cap P)\le N(B'(P_0)-B(P_0)).
  \end{equation*}
This proves Lemma~\ref{lemma:alt}.
\end{proof}

The next lemma gives a basic property of peelings.

\begin{lemma}\label{lemma:peeling0}  If $P$ is a peeling for
$S$, then every connected component of $P'\subset \G'_P$ has a
nonempty boundary.
\end{lemma}
  \begin{proof} Let $P_0$ be a nonempty subset of $P'$.  From the
definition of peeling,
  \begin{equation*}
H(P_0\cap P)\le FR(T)(B'(P_0)-B(P_0)).
  \end{equation*}
If $B(P_0)=0$, then $P_0\subset P$, and so $H(P_0\cap P)\ne 0$.  It
easily follows that $B'(P_0)$ is positive for every nonempty subset
$P_0$ of $P'$.  In particular, it is positive if $P_0$ is a connected
component of $P'$.  So some edge of $\G'_P$ contains exactly one
vertex of $P_0$.  This proves Lemma~\ref{lemma:peeling0}.
\end{proof}

Peelings are important for two reasons:

\begin{theorem}\label{thm:peeling1}  If $P$ is a peeling
for $S=P\amalg T$, then every cooling function for $T$ rel $h_0$ can
be extended to a cooling function for $S$ rel $h_0$ without increasing
the cooling norm.
\end{theorem}

\begin{theorem}\label{thm:peeling2}  If $T$ is a proper
subset of $S$ which is F{\o}lner-optimal, then $S\setminus T$ is a
peeling for $S$.
\end{theorem}

These two theorems apply as follows: A subset $S$ can be \emph{peeled}
until a subset $S_0$ is obtained that admits no peeling. When that
happens, Theorem~\ref{thm:peeling2} implies that $FR(S_0)$ is as large
as the \F\ ratio of any of its subsets. Since $S$ admits a cooling
function rel $h_0$, so does $S_0$.  Hence, by the cooling theorem, the
set $S_0$ admits a cooling function $c_0$ of cooling norm
$FR(S_0)$. By Theorem~\ref{thm:peeling1}, this cooling function
extends to a cooling function for $S$ of norm $FR(S_0)$. We conclude
that $FR(S_0)$ maximizes the \F\ ratio of subsets of $S$, that
$FR(S_0)$ is the minimal cooling norm for $S$, and that the extension
of $c_0$ is an optimal cooling function for $S$. The following problems
remain:

{\bf Problems.} Give an efficient algorithm for determining when a
peeling exists. Give an efficient algorithm for finding a peeling when
a peeling exists.

\begin{proof}[Proof (\ref{thm:peeling1})] Suppose we are given a
cooling function for $T$ rel $h_0$. We can modify this function so
that it is constant on all boundary edges of $T$ and equal to the norm
$N$ of the function on those edges. It suffices to show that this
modified cooling function extends to a cooling function for $S$ rel
$h_0$, for returning the values of the extension on the boundary edges
of $T$ to their original values will not destroy the cooling
properties of the extension.

From this point, we concentrate on $P'\subset \G'_P$. We extend
$h_0|_P:P\to (0,\infty)$ to $P'$ by defining $h_0(x) = N$ for each
$x\in P'\setminus P$. Lemmas~\ref{lemma:peeling0} and
\ref{lemma:exists} combine to show that $P'$ admits a cooling function
rel $h_0$.  Hence the relative cooling theorem applies
to this situation.  It implies that the smallest possible cooling norm
for $P'$ rel $h_0$ is $M = \max\{FR(P_0)\,|\,P_0 \subset P'\}$.
Lemma~\ref{lemma:alt} implies that $M\le FR(T)\le N$.  Hence a cooling
function for $P'$ with norm $M$ provides an extension of the original
cooling function for $T$.  This proves Theorem~\ref{thm:peeling1}.
\end{proof}

\begin{proof}[Proof (\ref{thm:peeling2})] Let $T$ be a proper subset
of $S$ whose \F\ ratio $N=FR(T)$ is maximal among all subsets of
$S$. Let $P = S \setminus T$, a nonempty subset of $S$. We prove that
$P$ is a peeling for $S$. Indeed, statement 2 of the cooling theorem
implies that there exists a cooling function $c:E(S) \to \R$ for $S$
rel $h_0$ of norm $N$. This function also cools $T$ rel $h_0$. With
the edges of $\partial E(T)$ oriented away from $T$, the values of $c$
on them must be constant with value $N$, for, if some boundary value
carries less than the average $N$ required of the boundary edges of
$T$, then another would have to carry more than the average, a
contradiction.

We return to $\G'_P$ as in the proof of Theorem~\ref{thm:peeling1}.
We again extend $h_0|_P:P\to (0,\infty)$ to $P'$ by defining
$h_0(x)=N$ for each $x\in P'\setminus P$.  The edges of $E(P)$ are in
canonical bijective correspondence with the edges of $\G'_P$, and so
$c$ determines a function $c'$ from the edges of $\G'_P$ to $\R$.  It
is clear that $c'$ cools $P$ viewed as a subset of $\G'_P$, and $c'$
also cools $P'\setminus P$ because every edge of $\partial E(T)$
carries $N$ units of heat away from $T$.  So $c'$ is a cooling
function rel $h_0$ with norm $N$.  Now we apply statement 2 of the
relative cooling theorem to conclude that $FR(P_0)\le N=FR(T)$ for
every $P_0\subset P'$.  Now Lemma~\ref{lemma:alt} implies that $P$ is
a peeling for $S$.
\end{proof}

\section{Building optimal cooling functions for $\Z\oplus \Z$}
\label{sec:building} 
Our eventual goal (unfortunately, only partially
completed in this paper) is to find an efficient algorithm to find a
subset $S_0 \subset S$ having maximum \F\ ratio $FR(S_0)$. The steps
proposed are these:

(1) By peeling, find a subset $T\subset S$ that allows no further
peeling (or at least seems to allow no further peeling).

(2) Prove that $T$ is self-optimal ($FR(T) = \max\{FR(T_0)\,|\,
T_0\subset T\}$) by building a cooling function for $T$ of norm
$FR(T)$.

(3) Extend the cooling function on $T$ to a cooling function on $S$
without increasing the cooling norm.

The relative cooling theorem, as applied in
Theorem~\ref{thm:peeling1},  completes step~3.

In the remainder of this section, we show how these steps can be
carried out for the $n$-ball $B(n)$ in the free Abelian group
$\Z\oplus\Z$ with its standard two-generator generating set.  This
ball is shaped like a diamond. Asymptotically, the F{\o}lner-optimal
subset of $B(n)$ which we find is shaped like a regular octagon (stop
sign). To guess the optimal shape, we simply look at successive \F\
ratios as layers are removed at the corners: right, left, top, and
bottom. Our first task is to determine when one of these edge layers
is in fact a peeling. (Recall Section~\ref{sec:peelings}.) The
theorems that follow will imply that, in these special cases, we need
only compare a single fraction $m/2$ with the \F\ ratio of the set
before the deletion.

{\bf Setting:} We consider the Cayley graph $\G$ of $G=\Z^2$ with
standard generators and initial condition $h_0 \equiv 1$. We consider
$T \subset \Z^2$ finite, $T$ a subset of the lower open half-plane and
containing $([1,m] \times \{-1\})\cap \Z^2$ for some positive integer
$m$. We consider $P = ([1,m] \times \{0\})\cap \Z^2$. We ask whether
$P$ is a peeling for $P\cup T$.  We return to $P'\subset \G'_P$ as
in Section~\ref{sec:peelings}.

\begin{theorem}\label{thm:lattice1}
(1) If $P_0\subset P'\setminus P$, then $H(P_0\cap P)=0$ and
$B'(P_0)-B(P_0)=0$.

(2) If $P_0\subset P'$ and $P_0\not\subset P'\setminus P$, then
$B'(P_0)-B(P_0)>0$ and the maximum 
  \begin{equation*}
\frac{H(P_0\cap P)}{B'(P_0)-B(P_0)}
  \end{equation*}
is realized by the set $P_0 = P'$ with value $m/2$. 

\end{theorem}

\begin{corollary}\label{cor:lattice2}  The set $P$ is a
peeling for $S =T\cup P$ if and only if $m/2\le FR(T)$.\end{corollary}

\begin{remark}Of course, the search for a peeling $P$ begins not with
$T$ but with $S = T \cup P$. We shall show later that this necessary
inequality is equivalent to the inequality with $FR(S)$ replacing
$FR(T)$ so that, as $P$ changes, we need not calculate the \F\ ratio
of the new $T$ but can instead retain the previously calculated old
ratio $FR(S)$.\end{remark}

\begin{proof}[Proof (Theorem~\ref{thm:lattice1})] Statement 1 and the
inequality $B'(P_0)-B(P_0)>0$ in statement 2 are easy to check.  So
let $P_0\subset P'$ with $P_0\not\subset P'\setminus P$.  We alter
$P_0$ by successive moves that can only increase the fraction being
maximized.

There is a canonical injective graph morphism from $\G'_P$ to $\G$,
which we use to identify $\G'_P$ with a subgraph of $\G$.

Without changing $B(P_0)$ and $H(P_0\cap P)$, we may permute the
elements of $P$ in such a way that all elements of $P_0\cap P$ appear
consecutively beginning at $(1,0)$. This can only decrease $B'(P_0)$
and thereby increase the fraction being maximized. 

If $(a,0)\in P_0$, then we may assume that $(a,-1)\in P_0$, for the
insertion of $(a,-1)$ into $P_0$ will increase $B(P_0)$ by $1$, will
decrease $B'(P_0)$ by $1$, and will leave $H(P_0\cap P)$ unchanged.

If $(a,-1)\in P_0$ but $(a,0)\notin P_0$, then we may assume such
points appear just after the pairs described in the paragraphs
above. Beginning with the left-most exemplar, we insert $(a,0)$ into
$P_0$. This increases $H$ by $1$ and leaves $B(P_0)$ unchanged.
Because $P_0\not\subset P'\setminus P$, it also leaves $B'(P_0)$
unchanged.

If some $(a,0) \in P$ does not appear in $P_0$, then beginning with
the leftmost such $a$, we insert both $(a,0)$ and $(a,-1)$. This
increases $H$ by $1$, $B(P_0)$ by $1$, and $B'(P_0)$ by $1$.

This easily proves Theorem~\ref{thm:lattice1}.
\end{proof}

We proceed to the improvement that allows us to substitute $FR(S)$ for
$FR(T)$. Let $n$ be a nonnegative integer, and let
$k\in\{0,\ldots,n\}$.  We let $B(n,k)$ denote the set remaining after
the first $k$ horizontal layers are removed from $B(n)$ both from top
and bottom and the first $k$ vertical layers are removed from $B(n)$
from both the left and right. The set $B(n,k)$ is an octagon until $k$
reaches $\lceil n/2\rceil$, at which point $B(n,k)$ becomes a square.

It is helpful to allow the notation $*\in \{<,\,=,\,>\}$ for the next
two theorems.

\begin{theorem}\label{thm:lattice3} For each $*
\in\{<,=,>\}$ and for each $k$ with $0 \le k < \lceil n/2\rceil $,
$$\vbox{\halign{
 $#$\hfil&$#$\hfil\hskip .2in\cr
 FR(B(n,k+1))*FR(B(n,k))& \quad\text{ iff }\quad FR(B(n,k+1))*(2k+1)/2\cr
 & \quad\text{ iff }\quad FR(B(n,k))*(2k+1)/2. \cr
 }}$$
 \end{theorem}

\begin{proof} The set $B(n,k+1)$ is gotten from $B(n,k)$ by removing
four layers from $B(n,k)$.  Removing one layer removes $2k+1$ vertices
and decreases the number of boundary edges by 2.  So $B(n,k)$ has
$x=4(2k+1)$ more elements than $B(n,k+1)$ and $y=8$ more boundary
edges.  Hence if $FR(B(n,k))=\zg_k/\zd_k,$
then
  \begin{equation*}
\frac{\zg_k}{\zd_k}=\frac{\zg_{k+1}+x}{\zd_{k+1}+y}\quad \text{and}
\quad\frac{\zg_{k+1}}{\zd_{k+1}}=\frac{\zg_k-x}{\zd_k-y}.
  \end{equation*}
A bit of algebra now yields that
  \begin{equation*}
\frac{\zg_{k+1}}{\zd_{k+1}}*\frac{\zg_k}{\zd_k}\quad\text{iff}\quad
\frac{\zg_{k+1}}{\zd_{k+1}}*\frac{x}{y}\quad\text{iff}\quad
\frac{\zg_k}{\zd_k}*\frac{x}{y}.
  \end{equation*}
This proves Theorem~\ref{thm:lattice3}.
\end{proof}

\begin{theorem}\label{thm:lattice4} Let $k_0$ be the
smallest integer greater than
$$(2n-\sqrt{2n^2+2n+1})/2.$$ Then $0\le k_0\le n/2$ and $FR(B(n,k))$
increases monotonically on the interval $0\le k \le k_0$ and decreases
monotonically on the interval $k_0\le k \le\lceil n/2\rceil$.
\end{theorem}

\begin{proof} It is an easy matter to calculate $FR(B(n,k))$:
$$FR(B(n,k)) = \fr 2n^2 + 2n + 1 - 4k^2.8n+4-8k..$$
Thus
$$FR(B(n,k)) * \fr 2k+1.2.$$
if and only if
$$2n^2+2n+1-4k^2 \quad*\quad 8nk + 4k - 8k^2 + 4n + 2 - 4k$$
if and only if
$$4k^2 + (-8n)k + (2n^2 - 2n - 1) \quad*\quad 0.$$
Set $p(x)=4x^2-8nx+2n^2-2n-1$.  The smaller root of $p(x)$ is
$$\vbox{\halign{ $\displaystyle #$\hfil&$\displaystyle#$\hfil\cr \fr
8n- \sqrt{64n^2 - 16(2n^2 - 2n - 1)}.8. &= \fr 8n- \sqrt{32n^2 + 32n +
16}.8.\cr &= \fr 2n-\sqrt{2n^2+2n+1}.2.\cr }}.$$ Since the larger root
of $p(x)$ is positive and $p(-1)>0$, it follows that $k_0\ge 0$.  One
verifies that $p(\frac{n-1}{2})<0$.  If $n$ is odd, then this implies
that $k_0\le (n-1)/2<n/2$ and if $n$ is even, then $k_0\le n/2$.  Thus
$0\le k_0\le n/2$.  Moreover, from the last display we conclude that
$FR(B(n,k))>(2k+1)/2$ for $k <k_0$ and $FR(B(n,k))<(2k+1)/2$ for
$k_0\le k \le\lceil n/2\rceil$. (Note that since every term in $p(x)$
is even except the last, $p(x)$ has no integer roots.) Therefore the
desired result follows from Theorem~\ref{thm:lattice3}.
\end{proof}

\begin{theorem}\label{thm:lattice5} For every integer $k$
with $0\le k < k_0$, the set $B(n,k)\setminus B(n,k+1)$ is a peeling
for $B(n,k)$.
\end{theorem}

\begin{proof} Theorem~\ref{thm:lattice4} implies that $FR(B(n,k))$
increases monotonically on the interval $0\le k\le k_0$.  This and
Theorem~\ref{thm:lattice3} imply that $FR(B(n,k+1))\ge (2k+1)/2$ if
$0\le k<k_0$.  Let $P=B(n,k)\setminus B(n,k+1)$.  Because $k_0\le
n/2$, it follows that the graph $\G'_P$ is the disjoint union of four
graphs, one for each of the four layers in $P$.
Theorem~\ref{thm:lattice1} applies to each of these four layers, and
we see that its conclusion holds even for $P$.  As in
Corollary~\ref{cor:lattice2}, it follows that $P$ is a peeling for
$B(n,k)$ if and only if $m/2\le FR(B(n,k+1))$, where $m=2k+1$.  
This completes the proof of Theorem~\ref{thm:lattice5}.
\end{proof}

\begin{theorem}\label{thm:lattice6} The set $B(n,k_0)$ is
asymptotically a regular octagon.
\end{theorem}

\begin{proof} The set is clearly an octagon with two horizontal sides
of lengths approximately $2k_0$, two vertical sides of lengths
approximately $2k_0$, and four diagonal sides at angles of forty-five
degrees from the axes, each having (as yet unknown) approximate length
$\ell$. We want to show that $\ell \approx 2k_0$. Each diagonal side
is the hypotenuse of an isosceles right triangle whose legs are
approximately of length $n - 2k_0$. We calculate:
$$\fr n.k_0. \approx \fr 2n.2n-\sqrt{2n^2+2n+1}. = \fr
2. 2-\sqrt{2+2/n+1/n^2}.\approx \fr 2.2-\sqrt{2}..$$ Thus,
$$\fr \ell.k_0. \approx \sqrt{2}\fr (n-2k_0).k_0.\approx
\sqrt{2}\bigg(\fr 2.2-\sqrt{2}.-2\bigg)=2.$$We conclude that the
octagon is almost regular for large values of $n$.
\end{proof}

\begin{theorem}\label{thm:lattice7} Let $k_0$ be the
smallest integer greater than
$$(2n-\sqrt{2n^2+2n+1})/2.$$ Then $B(n,k_0)$ is a F{\o}lner-optimal
subset of $B(n)$. That is, the \F\ ratio of $B(n,k_0)$ is the maximal
\F\ ratio of subsets of $B(n)$. Every optimal cooling function on
$B(n,k_0)$ can be extended to a cooling function on $B(n)$ without
increasing the cooling norm.
\end{theorem}

\begin{remark} The following are the major steps in the proof.

(1) The set $B(n,\lceil n/2\rceil)$ is a square and is self-optimal
because every $xy$-rectangle in $\Z \oplus \Z$ is self-optimal.

(2) If $k_0\le k<\lceil n/2 \rceil$ and $B(n,k+1)$ is self-optimal,
then $B(n,k)$ is also self-optimal.

(3) If $0\le k<k_0$, then $B(n,k)\setminus B(n,k+1)$ is a peeling for
$B(n,k)$. Hence each cooling function on $B(n,k+1)$ can be extended to
a cooling function on $B(n,k)$ without increasing the cooling norm.

Steps (1) and (2) will show that $B(n,k_0)$ is self-optimal. Step~(3)
will then show that $B(n,k_0)$ is a F{\o}lner-optimal subset of
$B(n)$.

Step~(3) is Theorem~\ref{thm:lattice5}. Step (1) is easy and will be
carried out next.  Step~(2) requires a fair amount of further work
which we shall carry out after Step~(1).
\end{remark}

For both Step~(1) and Step~(2) of the proof, we will manipulate what
we call \emph{difference diagrams}. For the first time in this paper
we will make use of the fact that we are considering the Cayley graph
$\G = \G(G,C)$ of an infinite group $G$ with finite, torsion-free
generating set $C$. If $S$ is a finite subset of the set of vertices
$G$ of $\G$, and if $f:E(S)\to \R$ is a cooling function, then we
associate with $f$ a \emph{difference diagram} $D$ as follows:
$$D:S\times C \to \R$$
$$D(s,c) = f((s,c,sc))-f((sc\inv,c,s)).$$ That is, $D(s,c)$ gives the
net heat loss at $s$ along the orbit of $G$ defined by the group
generator $c$. The difference diagram is a refinement of the heat loss
function at a vertex that we used in previous sections. We call $D$ a
diagram since its values can be recorded at the vertices of the graph
$\G$ in a pictorial diagram. In order to make the relationship
completely clear with the heat loss function $h(s)$ defined earlier,
we note that
$$\vbox{\halign{
 $\displaystyle#$\hfil&$\displaystyle#$\hfil\cr
h(s) = \sum_{i(e) = s}f(e) - \sum_{t(e) = s}f(e) 
  =\sum_{c\in C}D(s,c).\cr
 }}$$
We need only this simple difference diagram in carrying out Step~(1)
of the proof. However, in Step~(2) we shall consider the difference
between difference diagrams defined on two different sets, one
containing the other. In order to do this, we shall have to extend by
$0$ the difference diagram on the smaller set so that the two
difference diagrams have the same domain of definition. This new
difference will be called a \emph{diagram of second differences}.

\begin{theorem}\label{thm:lattice8} The cooling function $f$
is determined by its difference diagram and the values of $f$ on $\bd
E(S)$.
\end{theorem}

\begin{corollary}\label{thm:lattice9} If $S$ is self-optimal
and $f:E(S)\to \R$ is a cooling function of norm $FR(S)$, then $f$ is
determined by its difference diagram.
\end{corollary}

\begin{proof}[Proof (\ref{thm:lattice8} and \ref{thm:lattice9})] The
corollary follows immediately from the theorem since, in the case of
self-optimal sets $S$, the boundary values of optimal cooling functions
are constant and equal to $FR(S)$.

In general, since every generator $c\in C$ has infinite order and the
set $S$ is finite, every $c$-orbit determines edges
$e_1,\,e_2,\,\ldots,e_n$, with the initial vertex $i(e_1)$ of $e_1$
and the terminal vertex $t(e_n)$ of $e_n$ lying in the
complement of $S$, all other vertices lying in $S$ and $t(e_i)=i(e_i)c$
for $i\in\{1,\ldots,n\}$. The edge $e_1$ is a boundary edge so that
$f(e_1)$ is given. Inductively, $f(e_{i+1}) = f(e_i)+D(t(e_i),c)$.
\end{proof}

\begin{remark}[(Important!)]  In the self-optimal case, it is important
to realize that heat loss is always outwards at the boundary. This
means that in the previous paragraph, the first value $f(e_1)= -FR(S)$
must be thought of as a negative number and the terminal value
$f(e_n)=FR(S)$ as a positive value, with intermediate values $f(e_i)$
increasing, not necessarily monotonically, from the one value to the
other. It is important that in its possibly oscillatory traverse from
the negative value to the positive value, the values $f(e_i)$ never
stray from the interval $[-FR(S),FR(S)]$. In terms of the difference
diagram, this means that the partial sums along initial segments of
the orbit remain in the interval $[0,2\cdot FR(S)]$, with the total
sum equalling $2\cdot FR(S)$.

In the simplest cases, along each orbit the difference diagram
exhibits a \emph{partition} of $[0,2\cdot FR(S)]$ ; that is, the
values of $D$ are positive and sum to $2\cdot FR(S)$. However, a
partition does not always suffice. When finding a F{\o}lner-optimal
subset of $B(n)$, we shall often have to employ an oscillatory
traverse along certain orbits with $D$ exhibiting both positive and
negative values. The first case where this seems to be necessary is
$B(8)$.
\end{remark}

\begin{theorem}\label{thm:lattice10} In $G = \Z \oplus \Z =
\left<a,b\,|\,aba\inv b\inv = 1\right>$, every $xy$-rectangle $R$ is
self-optimal.
\end{theorem}

\begin{proof} Suppose $R$ denotes the $xy$-rectangle 
$([1,m]\times[1,n])\cap \Z^2$. Then
$$FR(R) = \fr m\cdot n.2(m+n)..$$ We express our cooling function in
terms of the unit $1/(2(m+n))$ so that our difference diagram can have
integral entries. Boundary values are thus given (in terms of the
prescribed units) by $\pm m\cdot n$ and the difference diagram must
have initial partial sums that traverse from $0$ to $2\cdot m\cdot n$,
never leaving the interval $[0,2\cdot m\cdot n]$. The negative sign is
used near initial points of orbits, the positive sign near terminal
points of orbits. A difference diagram is given by the function
$$D((x,y),a)=2\cdot n \quad\text{and}\quad
D((x,y),b)= 2\cdot m.$$Since $D((x,y),a) + D((x,y),b) = 2\cdot
(m+n)$, each vertex is cooled by exactly $2(m+n)$ fractional units,
hence by exactly $1$ real unit. Along $a$-orbits and $b$-orbits, the
difference diagram partitions $2\cdot m\cdot n$. Hence, the cooling
norm is $FR(R)$.
\end{proof}

\begin{theorem}\label{thm:lattice11} The set $B(n,k)$ is
self-optimal for all integers $k$ in the interval $[k_0,\lceil
n/2\rceil]$.
\end{theorem}

\begin{proof} Since the $xy$-square $B(n,\lceil n/2\rceil)$ is
self-optimal by the previous theorem, it is clearly enough to prove
that, if $\lceil n/2\rceil > k \ge k_0$ and $B(n,k+1)$ is
self-optimal, then $B(n,k)$ is also self-optimal.  Let $k$ be an
integer in the interval $[k_0,\lceil n/2\rceil-1]$.

We introduce $B(n,k+1,i)$, for $i\in \{0,1,2,3,4\}$, as the set formed
from $B(n,k+1)$ by adding $i$ of the layers of $B(n,k)$ not in
$B(n,k+1)$.  We let $FR(B(n,k+1,i)=\za_i/\zb_i$.  It suffices
to show that, if $B(n,k+1,i)$ is self-optimal, then $B(n,k+1,i+1)$ is
self-optimal, for $i\in \{0,1,2,3\}$. We shall need the
inequalities given in the following lemma.

\begin{lemma}\label{lemma:lattice12} If $j\in
\{0,1,2,3,4\}$, then
  \begin{equation*}
(2k+1)(\zb_j-2(2k+1))\le 2\za_j<(2k+1)\zb_j.
  \end{equation*}
\end{lemma}

\begin{proof} To prove the second inequality, suppose that
$FR(B(n,k))=\zg_k/\zd_k$, as in the proof of
Theorem~\ref{thm:lattice3}.  Because $k_0\le k\le
\left\lceil\frac{n}{2}\right\rceil-1$, Theorems~\ref{thm:lattice3} and
\ref{thm:lattice4} combine to imply that $\zg_k/\zd_k<(2k+1)/2$.
Hence
  \begin{equation*}
2\za_j=2(\zg_k+j(2k+1))<(2k+1)(\zd_k+2j)=(2k+1)\zb_j.
  \end{equation*}
This proves the second inequality.

We prove the first inequality first for the case $j = 4$, where
$B(n,k+1,4) = B(n,k)$. The inequality is then equivalent to the
inequality
  \begin{equation*}
(2k+1)(\zb_4-2(2k+1))\le 2\za_4.
  \end{equation*}
Substituting known values for $\za_4 = |B(n,k)|$ and $\zb_4 = |\bd
E(B(n,k))|$, as in the proof of Theorem~\ref{thm:lattice4}, we seek to
prove that
  \begin{equation*}
(2k+1)(8n+4-8k-2(2k+1))\le 2(2n^2+2n+1-4k^2).
  \end{equation*}
 This reduces to
$$4k^2 + (2-4n)k+(n^2-n)\ge 0,$$ equivalently,
  \begin{equation*}
(2k-n)(2k-(n-1))\ge 0.
  \end{equation*}
Since $k \le \lceil n/2\rceil -1\le (n-1)/2$, the proof is complete
for $j = 4$.

Now we induct downward for $j= 3,2,1,0$. In the original inequality,
before manipulation, each move downward subtracts exactly $2(2k+1)$
from each side of the inequality. This completes the proof.
\end{proof}

Assuming inductively that $B(n,k+1,i)$ is self-optimal, we let $D$
denote the difference diagram of some optimal cooling function for
$B(n,k+1,i)$. We employ $1/\zb_i$ as unit and assume inductively that
we can choose the entries of $D$ to be integers. We extend this
difference diagram by $0$ to the larger set $B(n,k+1,i+1)\times C$.

We seek a difference diagram $D_+$, expressed in units $1/\zb_{i+1}=
1/(\zb_i+2)$, that will exhibit $B(n,k+1,i+1)$ as self-optimal and ask
ourselves the properties that must be satisfied by the second
difference $E = D_+ - D$. Even though the two diagrams are expressed
in different units, each will have integer entries, and we express $E$
as the integer difference of those entries. In other words, we
manipulate the number of heat loss \emph{units} required and compare
the number of units in the two diagrams.

We label the rows of $B = B(n,k+1,i+1)$ by symbols $R_0$, $R_1$,
$\ldots$ from top to bottom and assume that $R_0$ is the new row with
$2k+1$ elements.

We label the columns of $B$ that contain an element of $R_0$ by
symbols $C_1$, $\ldots$, $C_{2k+1}$ from left to right and call these
columns the \emph{central columns}.

We label the other columns of $B$ by symbols $S_1$, $S_2$, $\ldots$
from left to right and call these columns the \emph{side columns}.

We view the elements of $B$ as cells for holding heat loss tokens.
Our task is to deposit heat loss tokens in these cells subject to the
following conditions.

(1) Row $R_0$ is assigned $2\cdot \za_{i+1}$ tokens. This new row
$R_0$ had no heat loss at all assigned by the cooling function for
$B(n,k+1,i)$; hence the difference function $E$ must account for the
total difference $2\cdot \za_{i+1}$ required along every orbit. All
other rows and columns are assigned only $2(2k+1)$ tokens. This
requirement reflects the fact that the old rows and columns require
exactly $2(\za_{i+1} - \za_{i})= 2(2k+1)$ more heat loss units than
assigned to that row or column by the cooling function for
$B(n,k+1,i)$.

(2) After distribution, each cell of $R_0$ is to have $\zb_{i+1}$
tokens. At these vertices of the new orbit $R_0$, the total heat loss
recorded by $E$ must be $\zb_{i+1}$ since there was no heat loss there
with the cooling function for $B(n,k+1,i)$. All other cells are to
have $2$ tokens. This requirement reflects the fact that the old
vertices require only $\zb_{i+1} - \zb_{i} = 2$ more heat loss units
than assigned to them by the cooling function for $B(n,k+1,i)$.

{\bf Heat-loss token distribution in the new row $R_0$:} Dividing, we find
$$2\cdot \za_{i+1} = (2k+1)q + r,$$where the quotient $q$ and the
remainder $r$ are integers and $0 \le r < 2k+1$. That $q < \zb_{i+1}$
follows from the second inequality of Lemma~\ref{lemma:lattice12} with
$j=i+1$.  That $\zb_{i+1}-q \le 2(2k+1)$ follows from the first
inequality of Lemma~\ref{lemma:lattice12} with $j=i+1$.

Thus we may place $q$ of the tokens assigned to the row $R_0$ into
each cell of $R_0$, with one extra in each of the first $r$ cells. We
may then place either $\zb_{i+1} - q$ (or $\zb_{i+1} - q - 1$, as
appropriate) column tokens into the cells of $R_0$ so that the total
number of tokens in each cell is $\zb_{i+1}$, as required.

{\bf Distribution of the heat-loss tokens associated with the side
columns:} In side column $S_j$, place $2$ of the tokens assigned
to $S_j$ into each of the bottom $2k+1$ cells of $S_j$.

{\bf The critical region:} Almost always there will be cells of the
side columns which still contain no heat loss tokens. These cells will
form two triangular regions. These two triangular regions together
with the portions of the central columns at the same height form what
we call the \emph{critical region}. See Figure~\ref{fig:region}.

  \begin{figure}
\centering
\includegraphics{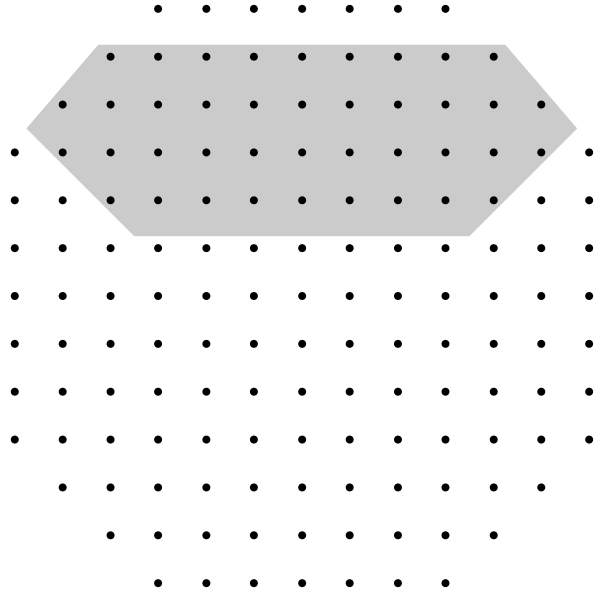} \caption{ The set
$B(9,4,4)=B(9,3)$, whose critical region is within the shaded polygon.}
\label{fig:region}
  \end{figure}

{\bf Distribution in the central columns beneath the critical region:}
Consider a row $R_j$ that lies beneath the critical region. The cells
in that row which have not yet been filled with two heat loss tokens
are precisely those in the central columns. Place $2$ of the heat
loss tokens assigned to $R_j$ into each of those cells. This will use up
all of the $2(2k+1)$ heat loss tokens assigned to that row.

{\bf Distribution of heat-loss tokens in the critical region:} Fill
that portion of the triangular regions in row $R_j$ with 2 tokens
assigned to $R_j$ in each cell.  For large $n$, the number of tokens
needed for this distribution exceeds the number, $2(2k+1)$, of tokens
assigned to $R_j$.  For each extra token needed, we create a new heat
loss token and a compensating \emph{negative} heat loss token.  These
negative heat loss tokens will be distributed in the intersection of
the row $R_j$ with the central columns.

In this paragraph we verify that the partial sums for $E$ along
initial segments of $R_j$ stay in the range $[0,2\cdot (2k+1)]$.  It
will then follow that the partial sums along initial segments of $R_j$
for the associated cooling function stay in the range
$[-FR(B(n,k+1,i+1)), FR(B(n,k+1,i+1))]$.  The number of side columns
left of the central columns which intersect $R_j$ in the critical
region is at most $n-2k-1$.  In order for the partial sums for $E$
along initial segments of $R_j$ to be at most $2\cdot (2k+1)$, we
need $n-2k-1\le 2k+1$.  This is equivalent to $k\ge (n-2)/4$.  As in
the proof of Theorem~\ref{thm:lattice4}, we let
$p(x)=4x^2-8nx+2n^2-2n-1$, and we verify that $p(n)<0$ and
  \begin{equation*}
p\left(\frac{n-2}{4}\right)=\frac{(n-2)^2}{4}+2n-1>0.
  \end{equation*}

Since $k$ is between the smaller root of $p(x)$ and $n$, it follows
that $k\ge (n-2)/4$.  This verifies that our partial sums are at most
$2\cdot (2k+1)$.  Symmetry shows that they are nonnegative.  Thus the
partial sums for $E$ along initial segments of $R_j$ stay in the range
$[0,2\cdot (2k+1)]$.

It remains only to distribute the remaining positive and negative row
tokens and the remaining central column tokens to fill the central
portion of the critical region.

We put all of the remaining column tokens in a single pot. We know
exactly how many row tokens remain to be placed in each row, and in
each row either every remaining token is positive or every remaining
token is negative. This tells us exactly how many column tokens need
to be placed in each row. We assign that number of column tokens to
each row.  Three paragraphs below we show that the pot contains
exactly the correct number of tokens to do this.

We distribute the column tokens assigned to row $R_1$ as follows. We
place the first column token assigned to row $R_1$ in column $C_1$, the
second in column $C_2$, the third in column $C_3$, and so on,
continuing to column $C_1$ if necessary, until the column tokens
assigned to $R_1$ have been exhausted.

Augment row and column numbers by 1 from where the last
column token was inserted. Place the column tokens assigned
to that row in successive columns. Iterate.  When we are done, every
column will contain the same number, $2k+1$, of column tokens.  Then
we distribute the remaining row tokens so that the total value of the
tokens in each of these cells is 2.

Now we show that the pot contains the correct number of tokens.
Suppose that all of the tokens are distributed, possibly leaving some
of these cells without the required two tokens.  After distributing
all of the tokens, the heat loss along every row and column is $2
\za_{i+1}$ units.  There are $\zb_{i+1}/2$ rows and columns, so the
total number of units is $\za_{i+1}\zb_{i+1}$.  Since there are
$\za_{i+1}$ cells, the average number of units per cell is
$\zb_{i+1}$.  It follows that every cell has received the needed
number of tokens.  Similarly, if every cell has received the needed
number of tokens, then the average heat loss along every row and
column is $2 \za_{i+1}$ units.  This means that no tokens remain.  The
construction is then complete.

It remains only to note that the column distributions are partitions
(no negative values for $E$). Although the row sums oscillate because
of the negative tokens, again the norm stays in the desired
bounds. Hence the associated cooling function has the correct norm.

Since we have constructed the required second difference $E$, we
conclude that it is possible to construct a cooling function for
$B(n,k+1,i+1)$ that has norm $FR(B(n,k+1,i+1)$, and the proof is
complete.  
\end{proof}

\end{document}